\documentclass[reqno,11pt]{amsart}
\usepackage{amssymb,amsmath}
\usepackage{tikz,enumerate,url}
\usepackage{color}

\numberwithin{equation}{section}
\newtheorem{theorem}[equation]{Theorem}

\newtheorem{proposition}[equation]{Proposition}

\theoremstyle{definition}

\newtheorem{example}[equation]{Example}
\newtheorem*{remark}{Remark}

\title[Linear recurrence sequences with indices in arithmetic progression]{%
Linear recurrence sequences with indices in arithmetic progression and their sums}
\author{Daniel Birmajer}
\address{Department of Mathematics\\ Nazareth College\\ 4245 East Ave.\\ Rochester, NY 14618}
\author{Juan B. Gil}
\address{Penn State Altoona\\ 3000 Ivyside Park\\ Altoona, PA 16601}
\author{Michael D. Weiner}


\begin{document}
\maketitle

\begin{abstract}
For an arbitrary homogeneous linear recurrence sequence of order $d$ with constant coefficients, we derive recurrence relations for all subsequences with indices in arithmetic progression. The coefficients of these recurrences are given explicitly in terms of partial Bell polynomials that depend on at most $d-1$ terms of the generalized Lucas sequence associated with the given recurrence. We also provide an elegant formula for the partial sums of such sequences and illustrate all of our results with examples of various orders, including common generalizations of the Fibonacci numbers.
\end{abstract}

\section{Introduction}

Let  $d$ be a positive integer and let $(a_n)$ be a sequence satisfying the recurrence relation
\begin{equation}\label{eq:recSequence}
 a_n = c_1a_{n-1}+\dots+c_da_{n-d} \;\text{ for } n\ge d, \; c_d\not=0.
\end{equation}
While it is not surprising that any subsequence of the form $(a_{mn+r})_{n\in\mathbb{N}}$, for fixed $m\in\mathbb{N}$ and $r\in\mathbb{N}_0=\mathbb{N}\cup\{0\}$, also satisfies a linear recurrence relation of order $d$, little is actually known about the structure of the coefficients of these recurrences. In this paper, we answer this question in full generality and give explicit formulas in terms of partial Bell polynomials in the coefficients $c_1,\dots,c_d$ of the original recurrence relation.

To this end, we introduce the associated sequence
\begin{equation}\label{eq:LucasTransform}
 \hat a_0=d, \quad 
 \hat a_n=\sum_{k=1}^n \frac{(k-1)!}{(n-1)!}B_{n,k}(1!c_1,2!c_2,\dots,d!c_d,0,\dots) \text{ for } n\ge 1,
\end{equation}
where $B_{n,k}=B_{n,k}(x_1,x_2,\dots)$ denotes the $(n,k)$-th partial Bell polynomial in the variables $x_1,x_2,\dots,x_{n-k+1}$. These polynomials, introduced by Bell \cite{Bell}, provide an efficient tool to work with linear recurrence sequences and their convolutions. For the definition and basic properties, see e.g.\ \cite[Section~3.3]{Comtet}.

It can be shown that $(\hat a_n)$ satisfies the same recurrence relation as $(a_n)$. For the special case of the Fibonacci sequence $(F_n)$, where $d=2$ and $c_1=c_2=1$, the associated sequence $(\hat F_n)$ is given by
\begin{equation*}
 \hat F_0=2, \quad \hat F_n=\sum_{k=1}^n \frac{(k-1)!}{(n-1)!}B_{n,k}(1!,2!,0,\dots) 
 = \sum_{k=0}^{\lfloor \frac{n}{2}\rfloor} \frac{n}{n-k}\binom{n-k}{k}.
\end{equation*}
This is precisely the Lucas sequence \cite[A000032]{Sloane}. Moreover, if $(a_n)$ is the generalized Fibonacci sequence of order $d$ (with\ $c_1=c_2=\cdots=c_d=1$), then $(\hat a_n)$ is the corresponding generalized Lucas sequence studied in \cite{HB77}. For this reason, we call the sequence defined by \eqref{eq:LucasTransform} the {\em Lucas transform} of $(c_1,\dots,c_d)$. One of the main features of $\hat a_n$ is that it can be written as
\begin{equation} \label{eq:LucasBinet}
 \hat a_n = \sum_{j=1}^d \alpha_j^n \;\text{ for } n\ge 0,
\end{equation}
where the $\alpha_j$'s are such that $(1-\alpha_1 t)\cdots(1-\alpha_d t)=1-c_1t-\cdots -c_d t^d$. The equivalence of \eqref{eq:LucasTransform} and \eqref{eq:LucasBinet} was observed by the authors in \cite{BGW14b}.

The main result of this paper (see Theorem~\ref{thm:arithmetic}) is that for an arbitrary linear recurrence sequence with constant coefficients $c_1,\dots,c_d$, as given in \eqref{eq:recSequence}, and for any fixed $m\in\mathbb{N}$ and $r\in\mathbb{N}_0$, the subsequence $(a_{mn+r})_{n\in\mathbb{N}}$ satisfies the linear recurrence relation
\begin{equation*}
 a_{mn+r} = \gamma_1\, a_{m(n-1)+r}+\gamma_2\, a_{m(n-2)+r}+\cdots+\gamma_d\, a_{m(n-d)+r} 
  \;\text{ for } n\ge d,
\end{equation*}
with $\gamma_k = \sum_{j=1}^k \frac{(-1)^{j+1}}{k!} B_{k,j}(0!\hat a_m,1!\hat a_{2m},\dots,(k-j)!\hat a_{(k-j+1)m})$ for $k=1,\dots,d$, where $(\hat a_n)$ is the Lucas transform of $(c_1,\dots,c_d)$.

In Section~\ref{sec:arithmetic}, we will prove this result and will illustrate our formula with examples of recurrences of order 2 and 3. We will also consider convolved Fibonacci sequences whose characteristic polynomials have roots of higher multiplicity. For brevity in our exposition, the number of examples discussed in this section is rather limited. However, all of the results presented in this paper are valid for homogeneous linear recurrence sequences of arbitrary order with constant coefficients over any integral domain.

In Section~\ref{sec:sums}, we turn our attention to the partial sums of a general linear recurrence sequence $(a_n)$ with characteristic polynomial $q(t)=1-c_1t-\dots-c_dt^d$, and give an elegant formula for $\sum_{j=0}^n a_j$ in terms of $a_{n+1},\dots, a_{n+d}$, see Theorem~\ref{thm:an_sum}. To this end, we first consider the sequence $(y_n)$ with generating function $1/q(t)$ and find a formula for its partial sums. The sequence $(y_n)$ is the {\sc invert} transform of $(c_1,\dots,c_d)$, and together with the sequences with generating functions $t^j/q(t)$ for $j=1,\dots,d-1$, they generate a basis for the space of linear recurrence sequences of order $d$ with coefficients $c_1,\dots,c_d$, cf.\ \cite{BGW14b} or \cite{Yang08}. The formula provided in Theorem~\ref{thm:an_sum} is carried out for several basic examples.

Because of the explicit nature of our two theorems, they can be easily combined to find formulas for sums of the form $\sum_{j=0}^n a_{mj+r}$. This is discussed at the end of Section~\ref{sec:sums} for recurrence sequences of order 2 and 3. For illustration purposes, we finish the paper with a few examples concerning the Tribonacci sequence. 
 
\section{Indices in arithmetic progression}
\label{sec:arithmetic}

Let $(a_n)$ be a sequence satisfying the recurrence relation \eqref{eq:recSequence}, and let $(\hat a_n)$ be the Lucas transform of the coefficients $(c_1,\dots,c_d)$, as defined in \eqref{eq:LucasTransform}. We start this section by showing that $\hat a_n$ admits the representation \eqref{eq:LucasBinet}. Let $\alpha_1,\dots,\alpha_d$ be defined by $(1-\alpha_1 t)\cdots(1-\alpha_d t)=1-c_1t-\cdots -c_d t^d$, and let $s_n = \sum_{j=1}^d \alpha_j^n$ for $n\ge 0$.

In \cite[Proposition~7]{BGW14b}, the authors showed that for $n\ge 1$,
\begin{equation*} 
 s_n=\sum_{k=1}^n (-1)^{n+k} \frac{(k-1)!}{(n-1)!}B_{n,k}(1!e_1,2!e_2,\dots,d!e_d,0,\dots),
\end{equation*}
where $e_1,\dots,e_d$ are the elementary symmetric functions in $\alpha_1,\dots,\alpha_d$. Since $e_j=(-1)^{j+1}c_j$ for every $j=1,\dots,d$, the homogeneity properties of the partial Bell polynomials give 
\begin{equation*} 
 B_{n,k}(1!e_1,2!e_2,\dots,d!e_d,0,\dots) = (-1)^{n+k} B_{n,k}(1!c_1,2!c_2,\dots,d!c_d,0,\dots),
\end{equation*}
which implies
\begin{equation*} 
 s_n = \sum_{k=1}^n \frac{(k-1)!}{(n-1)!}B_{n,k}(1!c_1,2!c_2,\dots,d!c_d,0,\dots) = \hat a_n \text{ for } n\ge 1.
\end{equation*}
Since $s_0=\hat a_0$, we conclude that $s_n=\hat a_n$ for all $n$, as stated in the introduction. Using the representation \eqref{eq:LucasBinet}, it is clear that $(\hat a_n)$ satisfies the same recurrence relation as $(a_n)$.

\begin{theorem}\label{thm:arithmetic}
Let $(a_n)$ be a linear recurrence sequence of order $d\ge 1$, satisfying the relation 
$a_n = c_1a_{n-1}+\dots+c_da_{n-d}$ for $n\ge d$, $c_d\not=0$. Let $(\hat a_n)$ be the Lucas transform of $(c_1,\dots,c_d)$. For any fixed $m\in\mathbb{N}$ and $r\in\mathbb{N}_0$, the subsequence $(a_{mn+r})_{n\in\mathbb{N}}$ satisfies the linear recurrence relation
\begin{equation*}
 a_{mn+r} = \gamma_1\, a_{m(n-1)+r}+\gamma_2\, a_{m(n-2)+r}+\cdots+\gamma_d\, a_{m(n-d)+r} 
 \;\text{ for } n\ge d,
\end{equation*}
where each $\gamma_k$ is given by
\begin{equation*}
 \gamma_k = \sum_{j=1}^k \frac{(-1)^{j+1}}{k!} B_{k,j}(0!\hat a_m,1!\hat a_{2m},\dots,(k-j)!\hat a_{(k-j+1)m}).
\end{equation*}
\end{theorem}

\begin{proof}
Since the sequences $(a_{mn+r})_{n\in\mathbb{N}}$ and $(\hat a_{mn})_{n\in\mathbb{N}}$ satisfy the same recurrence relation, it suffices to consider the latter. Using the representation \eqref{eq:LucasBinet}, for $m\in\mathbb{N}$, we get
\begin{equation*}
 \hat a_{mn} = \sum_{j=1}^d \alpha_j^{mn} = \sum_{j=1}^d (\alpha_j^{m})^n,
\end{equation*}
thus for $n\ge d$, $(\hat a_{mn})_{n\in\mathbb{N}}$ satisfies the recurrence relation
\begin{equation*}
 \hat a_{mn} = e_1^{(m)}\hat a_{m(n-1)}-e_2^{(m)} \hat a_{m(n-2)}+\cdots+(-1)^{d+1}e_d^{(m)} \hat a_{m(n-d)},
\end{equation*}
where $e_1^{(m)},\dots,e_d^{(m)}$, are the elementary symmetric functions in $\alpha_1^m,\dots,\alpha_d^m$. 

For every $k = 1,\dots,d$, let $\gamma_k = (-1)^{k+1}e_k^{(m)}$. Once again, by  \cite[Proposition~7]{BGW14b}, we have
\begin{align*} 
 \hat a_{mn} 
  &=\sum_{k=1}^n (-1)^{n+k} \frac{(k-1)!}{(n-1)!}B_{n,k}(1!e_1^{(m)},2!e_2^{(m)},\dots,d!e_d^{(m)},0,\dots)\\
  &=\sum_{k=1}^n \frac{(k-1)!}{(n-1)!}B_{n,k}(1!\gamma_1,2!\gamma_2,\dots,d!\gamma_d,0,\dots),
\end{align*}
and therefore
\begin{equation*}
 (n-1)! \hat a_{mn} = \sum_{k=1}^n (k-1)! B_{n,k}(1!\gamma_1,2!\gamma_2,\dots,d!\gamma_d,0,\dots).
\end{equation*}
Finally, Lagrange inversion gives
\begin{equation*}
 \gamma_k = \sum_{j=1}^k \frac{(-1)^{j+1}}{k!} B_{k,j}(0!\hat a_m,1!\hat a_{2m},\dots,(k-j)!\hat a_{(k-j+1)m}).
\end{equation*}
This proves the claimed recurrence relation for the sequence $(\hat a_{mn})_{n\in\mathbb{N}}$, and therefore for any sequence of the form $(a_{mn+r})_{n\in\mathbb{N}}$.
\end{proof}

\begin{remark}
Clearly, $\gamma_1=\hat a_m$, and since $\gamma_d=(-1)^{d+1}e_d^{(m)}$, we have $\gamma_d = (-1)^{(d+1)(m+1)} c_d^m$.
\end{remark}

Here is a basic example:

\begin{example}[$k$-Fibonacci] \label{ex:kFibonacci}
For $k\in\mathbb{N}$ let $(F_{k,n})_{n\in\mathbb{N}}$ be the sequence defined by
\[ F_{k,0}=0,\; F_{k,1}=1, \,\text{ and }\, F_{k,n+1}=kF_{k,n}+F_{k,n-1} \; \text{ for } n\ge 1. \]
In this case, $(\widehat F_{k,n})_{n\in\mathbb{N}}$ is the $k$-Lucas sequence denoted by $(L_{k,n})_{n\in\mathbb{N}}$ in the existing literature (see e.g.\ \cite{FP09}). By means of Theorem~\ref{thm:arithmetic}, we then get
\begin{equation}\label{eq:kFibonacci}
 F_{k,mn+r} = L_{k,m}\, F_{k,m(n-1)+r}+(-1)^{m+1}F_{k,m(n-2)+r} \;\text{ for } n\ge 2.
\end{equation}
Moreover, the representation \eqref{eq:LucasTransform} gives the identity
\begin{equation*}
 L_{k,m} =\sum_{j=1}^m \frac{(j-1)!}{(m-1)!}B_{m,j}(1!k,2!, 0,\dots) 
 =\sum_{j=0}^{m-1} \frac{m}{m-j}\binom{m-j}{j}k^{m-2j}.
\end{equation*}
The recurrence relation \eqref{eq:kFibonacci} coincides with the one given in \cite[Lemma~3]{FP09}. It is easy to check that $L_{k,m}=F_{k,m-1}+F_{k,m+1}$.
\end{example}

\smallskip
\begin{remark}
An interesting consequence of Theorem~\ref{thm:arithmetic} is that the structure of the recurrence relation satisfied by any arithmetic subsequence of a given linear recurrence sequence with constant coefficients, only depends on the order of the given recurrence. For example, for any linear recurrence sequence $(a_n)$ of order 2 with coefficients $c_1, c_2$, we always have
\begin{equation*}
 a_{mn+r} = \hat a_m\, a_{m(n-1)+r} + (-1)^{m+1}c_2^m a_{m(n-2)+r} \;\text{ for } n\ge 2,
\end{equation*}  
and for a linear recurrence of order 3 with coefficients $c_1, c_2, c_3$, we get
\begin{equation*}
 a_{mn+r} = \hat a_m\, a_{m(n-1)+r} + \tfrac12(\hat a_{2m}-\hat a_m^{\,2})\, a_{m(n-2)+r} + c_3^m a_{m(n-3)+r}
 \;\text{ for } n\ge 3,
\end{equation*} 
where $(\hat a_n)$ is the Lucas transform of the coefficients of $(a_n)$. Thus the key is to understand the terms $\hat a_{m}, \hat a_{2m}, \dots, \hat a_{dm}$, for which the representation in terms of partial Bell polynomials may be useful.
\end{remark}

In order to illustrate the use of \eqref{eq:LucasTransform}, we now consider two examples of linear recurrence sequences of order three. They both use the following identity:
\begin{equation*}
 B_{m,j}(x_1,x_2,x_3, 0,\dots) 
 = \sum_{\ell=0}^j \tfrac{m!}{j!} \tbinom{j}{j-\ell} \tbinom{j-\ell}{m+\ell-2j}
 \left(\tfrac{x_1}{1!}\right)^{\ell}\left(\tfrac{x_2}{2!}\right)^{3j-m-2\ell}\left(\tfrac{x_3}{3!}\right)^{m-2j+\ell}.
\end{equation*}

\smallskip
\begin{example}[Tribonacci, {A000073} in \cite{Sloane}] \label{ex:tribonacci}
Let $(t_n)$ be defined by
\begin{gather*}
  t_0=t_1=0,\;\; t_2=1, \\
  t_n=t_{n-1}+t_{n-2}+t_{n-3} \;\text{ for } n\ge 3.
\end{gather*}
Theorem~\ref{thm:arithmetic} gives the recurrence relation
\begin{equation}\label{eq:tribonacci}
 t_{mn+r} = \hat t_m\, t_{m(n-1)+r}+\tfrac12(\hat t_{2m}-{\hat t_m}^{\,2}) t_{m(n-2)+r}
 + t_{m(n-3)+r} \;\text{ for } n\ge 3,
\end{equation}
where $(\hat t_m)$ is the Lucas transform of $(1,1,1)$. By \eqref{eq:LucasTransform}, we have 
\begin{equation*}
 \hat t_m =\sum_{j=1}^m \frac{(j-1)!}{(m-1)!}B_{m,j}(1!,2!,3!, 0,\dots) 
 =\sum_{j=0}^{m-1} \sum_{\ell=\lceil j/2\rceil}^{j} \frac{m}{m-j}\binom{m-j}{\ell} \binom{\ell}{j-\ell}.
\end{equation*}
This is sequence \cite[A001644]{Sloane} and can also be described by
\begin{equation*}
  \hat t_0=3,\; \hat t_1=1,\; \hat t_2=3, \,\text{ and }\,
  \hat t_n=\hat t_{n-1}+\hat t_{n-2}+\hat t_{n-3} \;\text{ for } n\ge 3.
\end{equation*}
The recurrence relation \eqref{eq:tribonacci} is consistent with the one obtained in \cite[Theorem~1]{IA13}.
\end{example}

\begin{example}[Padovan, {A000931} in \cite{Sloane}] \label{ex:Padovan}
Consider the sequence defined by
\begin{gather*}
  P_0=1, \; P_1=P_2=0, \\
  P_n=P_{n-2}+P_{n-3} \;\text{ for } n\ge 3.
\end{gather*}
Theorem~\ref{thm:arithmetic} gives the recurrence relation
\begin{equation}\label{eq:Padovan}
 P_{mn+r} = \widehat P_m\, P_{m(n-1)+r}+\tfrac12(\widehat P_{2m}-{\widehat P_m}^{\,2}) P_{m(n-2)+r}
 + P_{m(n-3)+r} \;\text{ for } n\ge 3,
\end{equation}
where $(\widehat P_n)$ is the Perrin sequence \cite[A001608]{Sloane}. It satisfies the same recurrence relation as $(P_n)$ but with initial values $\widehat P_0=3$, $\widehat P_1=0$, and $\widehat P_2=2$. Moreover, by \eqref{eq:LucasTransform}, we have
\begin{equation*}
 \widehat P_m =\sum_{j=1}^m \frac{(j-1)!}{(m-1)!}B_{m,j}(0,2!,3!, 0,\dots) 
 =\sum_{j=\lceil m/2\rceil}^{m-1} \frac{m}{m-j}\binom{m-j}{2j-m}.
\end{equation*}
\end{example}

\begin{example}[Narayana's cows sequence, {A000930} in \cite{Sloane}] \label{ex:Padovan}
Let $(N_n)$ be defined by
\begin{gather*}
  N_0=N_1=N_2=1, \\
  N_n=N_{n-1}+N_{n-3} \;\text{ for } n\ge 3.
\end{gather*}
Once again, by Theorem~\ref{thm:arithmetic}, we get the recurrence relation
\begin{equation}\label{eq:Narayana}
 N_{mn+r} = \widehat N_m\, N_{m(n-1)+r}+\tfrac12(\widehat N_{2m}-{\widehat N_m}^{\,2}) N_{m(n-2)+r}
 + N_{m(n-3)+r} \;\text{ for } n\ge 3,
\end{equation}
where 
\begin{equation*}
 \widehat N_m =\sum_{j=1}^m \frac{(j-1)!}{(m-1)!}B_{m,j}(1!,0,3!, 0,\dots) 
 =\sum_{j=0}^{\lfloor (m-1)/2 \rfloor} \frac{m}{m-2j} \binom{m-2j}{j}.
\end{equation*}
While $(N_n)$ counts the number of compositions of $n$ into parts 1 and 3, it can be shown that $(N_{3n-1})$ counts the number of $\binom{n+1}{2}$-color compositions of $n$. Since $\widehat N_3=4$ and $\widehat N_6=10$, this subsequence satisfies the relation
\begin{gather*}
 N_2=1,\; N_5=4, \; N_8=13, \\
 N_{3n+2} = 4 N_{3(n-1)+2}-3 N_{3(n-2)+2} + N_{3(n-3)+2} \;\text{ for } n\ge 3.
\end{gather*}
\end{example}

\medskip
We finish this section with a linear recurrence sequence of order 4 whose generating function has roots of multiplicity 2.
\begin{example}[Convolved Fibonacci, {A001629} in \cite{Sloane}]
Let $(a_n)$ be the sequence obtained by convolving the Fibonacci sequence with itself. This sequence can be described by
\begin{gather*}
  a_0=a_1=0,\;\, a_2=1, \; a_3=2,\\
  a_n=2a_{n-1}+a_{n-2}-2a_{n-3}-a_{n-4} \;\text{ for } n\ge 4.
\end{gather*}
In this case, the Lucas transform $\hat a_n$ of $(2,1,-2,-1)$ satisfies $\hat a_n = 2L_n$, where $(L_n)$ is the Lucas sequence \cite[A000032]{Sloane}. By Theorem~\ref{thm:arithmetic}, for $n\ge 4$ we then get
\begin{equation*}
 a_{mn} = \gamma_1 a_{m(n-1)}+ \gamma_2 a_{m(n-2)} + \gamma_3 a_{m(n-3)} +\gamma_4 a_{m(n-4)} 
\end{equation*}
with
\begin{align*}
 \gamma_1 &= \hat a_m = 2L_m, \quad  \gamma_4 = -1, \\
 \gamma_2 &= \tfrac12(\hat a_{2m}-\hat a_m^{2}) = L_{2m}-2L_m^{2} = 2(-1)^{m+1}-L_m^{2}, \\
 \gamma_3 &= \tfrac16(2\hat a_{3m}-3\hat a_m\hat a_{2m}+\hat a_m^{3})
 = \tfrac23(L_{3m}-3L_m L_{2m}+2L_m^{3}) = (-1)^m 2L_m.
\end{align*}
Here we have used the known identities $L_{2m}=L_m^2-2(-1)^m$ and $L_{3m}=L_m^3-3(-1)^m L_m$. In conclusion, for $n\ge 4$ we have
\begin{equation}
 a_{mn} = 2L_m\, a_{m(n-1)}-\big(2(-1)^m+L_m^{2}\big) a_{m(n-2)}+(-1)^{m}2L_m a_{m(n-3)} - a_{m(n-4)}.
\end{equation}
For the special cases $m=2,3,4,5$, we have $L_2=3$, $L_3=4$, $L_4=7$, $L_5=11$, and so
\begin{align*}
 a_{2n} &= 6 a_{2(n-1)} - 11 a_{2(n-2)} + 6 a_{2(n-3)} - a_{2(n-4)}, \\
 a_{3n} &= 8 a_{3(n-1)} - 14 a_{3(n-2)} - 8 a_{3(n-3)} - a_{3(n-4)}, \\
 a_{4n} &= 14 a_{4(n-1)} - 51 a_{4(n-2)} + 14 a_{4(n-3)} - a_{4(n-4)}, \\
 a_{5n} &= 22 a_{5(n-1)} - 119 a_{5(n-2)} - 22 a_{5(n-3)} - a_{5(n-4)}.
\end{align*}
\end{example}

\section{Sums of linear recurrence sequences}
\label{sec:sums}

For fixed $c_1,\dots,c_d$ with $c_d\not=0$, let
\begin{equation} \label{eq:charpolynomial}
 q(t) = 1-c_1t-c_2t^2-\dots -c_dt^d,
\end{equation}
and let $(y_n)$ be the sequence with generating function $Y(t)=1/q(t)$. Denoting $c_0=-1$, we then have
\[ 1 = q(t) Y(t) = \bigg(\!-\sum_{n=0}^d c_n t^n\bigg)\bigg(\sum_{n=0}^\infty y_n t^n\bigg), \]
which implies $\sum_{i=0}^n c_i y_{n-i} = 0$ for every $n\ge 1$. Therefore,
\begin{equation*}
 -1 = c_0 + \sum_{n=1}^d\bigg(\sum_{i=0}^n c_i y_{n-i}\bigg)
  = \sum_{n=0}^d \sum_{i=0}^n c_i y_{n-i} = \sum_{j=0}^d\bigg(\sum_{i=0}^j c_i\bigg) y_{d-j}
\end{equation*}
and so
\begin{equation}\label{eq:base}
 q(1) y_0 = -\bigg(\sum_{i=0}^d c_i\bigg) y_0 = 1+ \sum_{j=0}^{d-1}\bigg(\sum_{i=0}^j c_i\bigg) y_{d-j}
 = 1 + \sum_{j=0}^{d-1}\bigg(\sum_{i=0}^{d-1-j} c_i\bigg) y_{j+1}.
\end{equation}
This is the base case for the following statement.

\begin{proposition}
Let $(y_n)$ be the linear recurrence sequence with generating function $1/q(t)$, where $q(t) = 1-c_1t-c_2t^2-\dots -c_dt^d$ with $c_d\not=0$, and let $c_0=-1$. Then for $n\ge 0$,
\begin{equation} \label{eq:yn_sum}
q(1)\sum_{j=0}^n y_j = 1+\sum_{j=0}^{d-1}\bigg(\sum_{i=0}^{d-1-j} c_i\bigg) y_{n+j+1}.
\end{equation}
\end{proposition}
\begin{proof} 
We proceed by induction on $n$. The base case $n=0$ was established in \eqref{eq:base}. Assume that \eqref{eq:yn_sum} holds for $n-1$. Then
\begin{align*}
 q(1)\sum_{j=0}^n y_j &= q(1)\sum_{j=0}^{n-1} y_j + q(1) y_n 
 = 1+\sum_{j=0}^{d-1}\bigg(\sum_{i=0}^{d-1-j} c_i\bigg) y_{n+j} + q(1) y_n \\
 &= 1+\sum_{j=1}^{d-1}\bigg(\sum_{i=0}^{d-1-j} c_i\bigg) y_{n+j} - c_dy_n
 = 1+\sum_{j=0}^{d-2}\bigg(\sum_{i=0}^{d-2-j} c_i\bigg) y_{n+j+1} - c_dy_n\\
 &= 1+\sum_{j=0}^{d-2}\bigg(\sum_{i=0}^{d-1-j} c_i\bigg) y_{n+j+1} -\sum_{j=0}^{d-2} c_{d-1-j} y_{n+j+1} - c_dy_n\\
 &= 1+\sum_{j=0}^{d-2}\bigg(\sum_{i=0}^{d-1-j} c_i\bigg) y_{n+j+1} - y_{n+d}
 = 1+\sum_{j=0}^{d-1}\bigg(\sum_{i=0}^{d-1-j} c_i\bigg) y_{n+j+1}.
\end{align*}
Hence the identity \eqref{eq:yn_sum} holds for all $n\ge 0$.
\end{proof}

Let $q(t)$ be as in \eqref{eq:charpolynomial}.  For $\ell\in\{0,1,\dots,d-1\}$ we let $\big(y_n^{(\ell)}\big)$ be the linear recurrence sequence with generating function $Y_\ell(t)=t^\ell/q(t)$. Note that $(y_n^{(0)})$ is the sequence $(y_n)$ introduced above, and for $\ell>0$ we have 
\[ y_0^{(\ell)}=\dots=y_{\ell-1}^{(\ell)}=0 \;\text{ and }\; y_n^{(\ell)}=y_{n-\ell} \;\text{ for } n\ge \ell. \] 
Clearly, the sequences $(y_n^{(0)})$, $(y_n^{(1)})$, \dots, $(y_n^{(d-1)})$ form a basis for the space of all linear recurrence sequences of order $d$ with coefficients $c_1,\dots, c_d$. 

More precisely, if $(a_n)$ is a linear recurrence sequence satisfying $a_n = c_1a_{n-1}+\dots+c_da_{n-d}$ with initial values $a_0,\dots, a_{d-1}$, then
\begin{equation} \label{eq:lambdas}
\begin{gathered}
 a_n=\lambda_0 y_n^{(0)}+\dots+\lambda_{d-1} y_n^{(d-1)}, \text{ where} \\
 \lambda_0=a_0 \text{ and } \lambda_n=a_n-\sum_{j=1}^n c_j a_{n-j} \text{ for } n=1,\dots,d-1.
\end{gathered}
\end{equation}

\begin{theorem} \label{thm:an_sum}
Let $(a_n)$ be a linear recurrence sequence of order $d$ satisfying
\begin{equation*}
 a_n = c_1a_{n-1}+\dots+c_da_{n-d} \;\text{ for } n\ge d,
\end{equation*}
with initial values $a_0,\dots, a_{d-1}$, and let $c_0=-1$. For $n\ge 0$, we have
\begin{equation*}
 q(1)\sum_{j=0}^n a_j = \sum_{j=0}^{d-1}\bigg(\sum_{i=0}^{d-1-j} c_i\bigg) \big(a_{n+j+1}-a_{j}\big),
\end{equation*}
where $q(1)=1-c_1-\dots-c_d$.
\end{theorem}
\begin{proof}
We start by writing $a_j=\lambda_0 y_j^{(0)}+\dots+\lambda_{d-1} y_j^{(d-1)}$ as in \eqref{eq:lambdas}. Thus
\begin{equation*}
 q(1)\sum_{j=0}^n a_j = q(1)\sum_{j=0}^n \sum_{\ell=0}^{d-1} \lambda_\ell y_{j}^{(\ell)}
 = q(1)\sum_{\ell=0}^{d-1} \lambda_\ell \bigg(\sum_{j=\ell}^n y_{j-\ell}\bigg)
 = \sum_{\ell=0}^{d-1} \lambda_\ell \bigg(q(1)\sum_{j=0}^{n-\ell} y_{j}\bigg), 
\end{equation*}
which by \eqref{eq:yn_sum} becomes
\begin{align*}
 q(1)\sum_{j=0}^n a_j
 &= \sum_{\ell=0}^{d-1} \lambda_\ell 
 \bigg(1+\sum_{j=0}^{d-1}\bigg(\sum_{i=0}^{d-1-j} c_i\bigg) y_{n+j+1-\ell}\bigg) \\
 &= \sum_{\ell=0}^{d-1} \lambda_\ell 
 + \sum_{\ell=0}^{d-1} \lambda_\ell \sum_{j=0}^{d-1}\bigg(\sum_{i=0}^{d-1-j} c_i\bigg) y_{n+j+1-\ell} \\
 &= \sum_{\ell=0}^{d-1} \lambda_\ell 
 + \sum_{j=0}^{d-1}\bigg(\sum_{i=0}^{d-1-j} c_i\bigg) \sum_{\ell=0}^{d-1} \lambda_\ell y_{n+j+1-\ell} 
 = \sum_{\ell=0}^{d-1} \lambda_\ell + \sum_{j=0}^{d-1} \bigg(\sum_{i=0}^{d-1-j} c_i\bigg) a_{n+j+1}.
\end{align*}
Now, by means of \eqref{eq:lambdas}, we have
\begin{equation*}
  \sum_{\ell=0}^{d-1} \lambda_\ell = -\sum_{j=0}^{d-1}\Big(\sum_{i=0}^j c_i\Big) a_{d-1-j}
  = -\sum_{j=0}^{d-1}\Big(\sum_{i=0}^{d-1-j} c_i\Big) a_{j},
\end{equation*}
and therefore,
\begin{equation*}
 q(1)\sum_{j=0}^n a_j = \sum_{j=0}^{d-1}\bigg(\sum_{i=0}^{d-1-j} c_i\bigg) \big(a_{n+j+1}-a_{j}\big),
\end{equation*}
as claimed.
\end{proof}

\begin{example}[$d$-step Fibonacci]
Let $d\in\mathbb{N}$ with $d\ge 2$. Let $(f^{(d)}_n)$ be defined by
\begin{equation*}
  f^{(d)}_0=\dots=f^{(d)}_{d-2}=0,\; f^{(d)}_{d-1}=1, \;\; f^{(d)}_n=f^{(d)}_{n-1}+\dots+f^{(d)}_{n-d} \;\text{ for } n\ge d.
\end{equation*}
By Theorem~\ref{thm:an_sum},
\begin{equation*}
 \sum_{j=0}^n f^{(d)}_j = \tfrac{1}{1-d}\sum_{j=0}^{d-1}(d-2-j) \Big(f^{(d)}_{n+j+1}-f^{(d)}_{j}\Big)
 = \tfrac{1}{1-d}\bigg(\sum_{j=0}^{d-1}(d-2-j) f^{(d)}_{n+j+1} + 1\bigg).
\end{equation*}
\end{example}

\begin{example}[$d$-step Lucas]
Let $d\in\mathbb{N}$ with $d\ge 2$. Let $(\ell^{(d)}_n)$ be the L-sequence associated with $(f^{(d)}_n)$. It satisfies the recurrence relation
\begin{gather*}
  \ell^{(d)}_0=d, \; \ell^{(d)}_{j}=2^j-1 \text{ for } j=1,\dots, d-1, \\
  \ell^{(d)}_n=\ell^{(d)}_{n-1}+\dots+\ell^{(d)}_{n-d} \;\text{ for } n\ge d.
\end{gather*}
By Theorem~\ref{thm:an_sum}, 
\begin{equation*}
 \sum_{j=0}^n \ell^{(d)}_j = \tfrac{1}{1-d}\sum_{j=0}^{d-1}(d-2-j) \Big(\ell^{(d)}_{n+j+1}-\ell^{(d)}_{j}\Big), 
\end{equation*}
which can be written as 
\begin{equation}
 \sum_{j=0}^n \ell^{(d)}_j = \tfrac{1}{1-d}\bigg(\sum_{j=0}^{d-1}(d-2-j) \ell^{(d)}_{n+j+1} - \frac{d(d-3)}{2} \bigg).
\end{equation}
In particular, for $d=2$ and $d=3$, we get
\begin{equation*}
 \sum_{j=0}^n \ell^{(2)}_j = \ell^{(2)}_{n+2} -1 \;\text{ and }\; 
 \sum_{j=0}^n \ell^{(3)}_j = \tfrac12\big(\ell^{(3)}_{n+3}-\ell^{(3)}_{n+1}\big)
 = \tfrac12\big(\ell^{(3)}_{n+2}+\ell^{(3)}_{n}\big),
\end{equation*}
which are sequences A001610 and A073728 in \cite{Sloane}, and for $d=4$,
\begin{equation*}
 \sum_{j=0}^n \ell^{(4)}_j = \tfrac13\big(\ell^{(4)}_{n+3}-\ell^{(4)}_{n+1} + \ell^{(4)}_{n}+2\big).
\end{equation*}
\end{example}

\medskip
\subsection*{Subsequences with indices in arithmetic progression}

As discussed in Theorem~\ref{thm:arithmetic}, given a linear recurrence sequence $(a_n)$ with constant coefficients, any subsequence of the form $(a_{mn+r})_{n\in\mathbb{N}}$ also satisfies a linear recurrence relation with constant coefficients that depend on $(\hat a_n)$, the Lucas transform of the coefficients of $(a_n)$. Consequently, Theorem~\ref{thm:an_sum} may be used to derive, in a straightforward manner, formulas for the sums $\sum_{j=0}^n a_{mj+r}$.

In order to illustrate the combined use of these theorems, we will discuss some examples for linear recurrences of order two and three. The higher the order of $(a_n)$, the more terms of the associated sequence $(\hat a_n)$ are required to find the coefficients of the recurrence relation satisfied by $(a_{mn+r})_{n\in\mathbb{N}}$. However, the number of terms needed is one less than the order. More precisely, if the order of $(a_n)$ is $d$, we will only need to compute $\hat a_m, \hat a_{2m},\dots, \hat a_{(d-1)m}$.  

\begin{example}[Linear recurrences of order 2] 
Let $(a_n)$ be defined by
\begin{equation*}
 a_n = c_1a_{n-1}+c_2a_{n-2} \;\text{ for } n\ge 2,
\end{equation*}  
with initial values $a_0$ and $a_1$. By Theorem~\ref{thm:arithmetic}, we know
\begin{equation*}
 a_{mn+r} = \hat a_m\, a_{m(n-1)+r} + (-1)^{m+1}c_2^m a_{m(n-2)+r} \;\text{ for } n\ge 2,
\end{equation*}  
where $\hat a_m$ is given by
\begin{equation*}
 \hat a_m = \sum_{j=1}^m \frac{(j-1)!}{(m-1)!}B_{m,j}(1!c_1,2!c_2,0,\dots)
 = \sum_{j=0}^{m-1} \frac{m}{m-j} \binom{m-j}{j} c_1^{m-2j}c_2^{j}.
\end{equation*}
Moreover, by Theorem~\ref{thm:an_sum}, 
\begin{align*}
 \sum_{j=0}^n a_{mj+r} &= \frac{\big(a_{m(n+2)+r}-a_{m+r}\big) 
 - (\hat a_{m}-1)\big(a_{m(n+1)+r} - a_{r}\big)}{\hat a_{m}+(-1)^{m+1}c_2^m - 1} \\
 &= \frac{a_{m(n+1)+r}-(-1)^{m}c_2^m\, a_{mn+r} + (\hat a_{m}-1)a_{r} - a_{m+r}}{\hat a_{m}-(-1)^{m}c_2^m - 1}.
\end{align*}

\medskip
For the special case of the $k$-Fibonacci sequence (cf.\ Example~\ref{ex:kFibonacci}), we get
\begin{align*}
 \sum_{j=0}^n F_{k,mj+r}
 &= \frac{F_{k,m(n+1)+r}-(-1)^{m}F_{k,mn+r} + (L_{k,m}-1)F_{k,r} - F_{k,m+r}}{L_{k,m}-(-1)^{m} - 1},
\intertext{and for the $k$-Lucas sequence, we have}
 \sum_{j=0}^n L_{k,mj+r} 
 &= \frac{L_{k,m(n+1)+r}-(-1)^{m}L_{k,mn+r} + (L_{k,m}-1)L_{k,r} - L_{k,m+r}}{L_{k,m}-(-1)^{m} - 1}.
\end{align*}
These formulas are consistent with the ones given in \cite{Falcon12,FP09}.
\end{example}

\begin{example}[Linear recurrences of order 3]
Let $(a_n)$ be defined by
\begin{equation*}
 a_n = c_1a_{n-1}+c_2a_{n-2}+c_3a_{n-3} \;\text{ for } n\ge 3,
\end{equation*}  
with initial values $a_0$, $a_1$, and $a_2$. By Theorem~\ref{thm:arithmetic}, we have
\begin{equation*}
 a_{mn+r} = \hat a_m\, a_{m(n-1)+r} + \tfrac12(\hat a_{2m}-\hat a_m^{\,2})\, a_{m(n-2)+r} + c_3^m a_{m(n-3)+r}
 \;\text{ for } n\ge 3,
\end{equation*} 
where $\hat a_m = \sum_{j=1}^m \tfrac{(j-1)!}{(m-1)!}B_{m,j}(1!c_1,2!c_2,3!c_3,0,\dots)$. Theorem~\ref{thm:an_sum} then gives
\begin{equation} \label{eq:sum3recurrence}
 \hat q(1)\sum_{j=0}^n a_{mj+r} 
 = \sum_{j=0}^{2}\bigg(\sum_{i=0}^{2-j} \hat c_i\bigg) \big(a_{m(n+j+1)+r}-a_{mj+r}\big),
\end{equation}
where $\hat c_0=-1$, $\hat c_1=\hat a_m$, $\hat c_2=\tfrac12(\hat a_{2m}-\hat a_m^{\,2})$, and $\hat q(1) = 1-\hat a_m - \tfrac12(\hat a_{2m}-\hat a_m^{\,2})-c_3^m$.

\medskip
For the special case of the Tribonacci sequence (cf.\ Example~\ref{ex:tribonacci})
\begin{equation*}
  t_0=t_1=0,\; t_2=1, \;\; t_n=t_{n-1}+t_{n-2}+t_{n-3} \;\text{ for } n\ge 3,
\end{equation*}
the above formula \eqref{eq:sum3recurrence} gives
\begin{equation*}
 \sum_{j=0}^n t_{mj+r} = \frac{t_{m(n+1)+r} + \big(1+\tfrac12(\hat t_{2m}-\hat t_m^{\,2})\big) t_{mn+r} 
 + t_{m(n-1)+r} +I_{m,r}}{\hat t_m + \tfrac12(\hat t_{2m}-\hat t_m^{\,2})},
\end{equation*}
where $I_{m,r}=\big(\hat t_m+\tfrac12(\hat t_{2m}-\hat t_m^{\,2})-1\big)t_{r} + (\hat t_m-1)t_{m+r} - t_{2m+r}$.
Here are a few values of the sequences $(t_n)$ and $(\hat t_n)$, taken from \cite{Sloane}:
\begin{align*}
\text{(A000073) } t_n: & \quad 0, 0, 1, 1, 2, 4, 7, 13, 24, 44, 81, 149, 274, 504, 927, \dots \\
\text{(A001644) } \hat t_n: & \quad 3, 1, 3, 7, 11, 21, 39, 71, 131, 241, 443, 815, 1499, 2757, \dots
\end{align*}
Tribonacci numbers have been extensively studied, and some special cases of the above formula can be found in the literature, see e.g.\ \cite{Kilic08} and \cite[Theorem~3]{IA13}.

\medskip
We finish this section with a short list of particular instances of the above sum. 
\begin{equation*}
  \sum_{j=0}^n t_j = \tfrac{1}{2}\big(t_{n+2}+t_{n}-1\big),
\end{equation*}
\begin{alignat*}{2} 
 \sum_{j=0}^n t_{2j} &= \tfrac12\big(t_{2n+1} + t_{2n}\big), & 
   \sum_{j=0}^n t_{2j+1} &= \tfrac12\big(t_{2n+2} + t_{2n+1} -1\big), \\
\sum_{j=0}^n t_{3j} &= \tfrac12\big(t_{3n+2} - t_{3n} -1\big), \quad & 
   \sum_{j=0}^n t_{4j} &= \tfrac14\big(t_{4n+2} + t_{4n} - 1\big),
\end{alignat*}
\begin{align*}
 \sum_{j=0}^n t_{5j+r} &= \tfrac{1}{22}\big(t_{5n+2+r} + 8t_{5n+1+r} + 5t_{5n+r} + I_{r}\big),
\end{align*}
where $I_0=-1$, $I_1=-9$, $I_2=7$, $I_3=-3$, and $I_4=-5$.
\end{example}

\bigskip
\subsection*{\sc Acknowledgement}
The authors would like to thank James Sellers for bringing sequences with indices in arithmetic progression to their attention.



\begin{thebibliography}{99}
\bibitem{Bell} E. T.~Bell, Exponential polynomials, {\it Ann. of Math.} \textbf{35} (1934), 258--277.
\bibitem{BGW14b} D.~Birmajer, J.~Gil, and M.~Weiner, \emph{Linear recurrence sequences and their convolutions via Bell polynomials}, J. Integer Seq. \textbf{18} (2015), no. 1, Article 15.1.2, 14 pp.
\bibitem{Comtet} L.~Comtet, {\it Advanced Combinatorics: The Art of Finite and Infinite Expansions}, D. Reidel Publishing Co., Dordrecht, 1974.
\bibitem{Falcon12} S.~Falcon \emph{On the $k$-Lucas numbers of arithmetic indexes}, Appl. Math. \textbf{3} (2012), 1202--1206.
\bibitem{FP09} S.~Falcon and A.~Plaza, \emph{On $k$-Fibonacci numbers of arithmetic indexes}, Appl. Math. Comput. \textbf{208} (2009), 180--185.
\bibitem{HB77} V.~Hoggatt and M.~Bicknell-Johnson, \emph{Generalized Lucas sequences}, Fibonacci Quart. \textbf{15} (1977), no. 2, 131--139.
\bibitem{IA13} N.~Irmak and M.~Alp, \emph{Tribonacci numbers with indices in arithmetic progression and their sums}, Miskolc Math. Notes \textbf{14} (2013), no. 1, 125--133.
\bibitem{Kilic08} E.~Kili\c{c}, \emph{Tribonacci sequences with certain indices and their sums}, Ars Combin. \textbf{86} (2008), 13--22.
\bibitem{Sloane} N. J. A.~Sloane, The On-Line Encyclopedia of Integer Sequences, \url{http://oeis.org}.
\bibitem{Yang08} S.~Yang, \emph{On the $k$-generalized Fibonacci numbers and high-order linear recurrence relations}, Appl. Math. Comput. \textbf{196} (2008), no. 2, 850--857.
\end{thebibliography}
\end{document}